\newif\ifabstr
\abstrfalse

\newif\ifcomments
\commentstrue 

\ifabstr 
\documentclass[runningheads]{llncs}
\usepackage[dvips]{epsfig}
\usepackage{amsmath,amssymb,graphicx}

\else 
\documentclass[12pt]{article}

\usepackage[latin1]{inputenc}
\usepackage{amsmath,amssymb,amsthm,graphicx}
\usepackage[dvips]{epsfig}
\usepackage{amsfonts}
\usepackage[left=1.4in, right=1.4in]{geometry}
\usepackage[usenames,dvipsnames]{color}

\usepackage[usenames,dvipsnames]{color}

\ifcomments

\else

\fi
\usepackage{url}

\input{xy}
\xyoption{all}

\newtheorem{theorem}{Theorem}

\newtheorem{lemma}[theorem]{Lemma}

\theoremstyle{remark}

\newtheorem*{remark*}{Remark}

\theoremstyle{definition}

\fi 

\newcommand{\R}{\mathbb{R}}

\newcommand{\Z}{\mathbb{Z}}

\newcommand{\marrow}{\marginpar{\boldmath$\longleftarrow$}}
\ifcomments
\newcommand{\martin}[1]{\ifhmode\newline\fi\marrow
\textsf{\textcolor{green}{\bf MARTIN:} #1\newline}}
\else
\newcommand{\martin}[1]{}
\fi

\long\def\comment#1\endcomment{}

\def\Int{\mathop{\fam0 Int}}
\def\t{\widetilde}

\DeclareMathOperator\NO{NO}
\def\Phin{\Phi_{\NO}}
\def\Kn{K_{\NO}}

\title{Hardness of almost embedding simplicial complexes in $\R^d$
}

\author{Arkadiy Skopenkov\thanks{Moscow Institute of Physics and Technology,
  Institutskiy per., Dolgoprudnyi, 141700, Russia,
  and Independent University of Moscow.
  B. Vlasy\-ev\-skiy, 11, Moscow, 119002, Russia. Research supported by the
  Russian Foundation for Basic Research Grant No. 15-01-06302, by Simons-IUM Fellowship and by the D. Zimin's Dynasty Foundation Grant.
  Email:
  \texttt{skopenko@mccme.ru}} \and Martin Tancer\thanks{
Department of Applied Mathematics, Charles University in Prague, Malostransk\'e n\'am. 25, 118 00,
Praha 1. Supported by the GA\v{C}R project 16-01602Y.
}}

\date{}

\begin{document}

\maketitle

\begin{abstract}
A map $f\colon K\to \R^d$ of a simplicial complex is an \emph{almost embedding} if $f(\sigma)\cap f(\tau)=\emptyset$ whenever $\sigma,\tau$ are disjoint simplices of $K$.

{\bf Theorem.} {\it Fix integers $d,k\ge2$ such that $d=\frac{3k}2+1$.

(a) Assume that $P\ne NP$.
Then there exists a finite $k$-dimensional complex $K$ that does not admit an almost embedding in $\R^d$ but for which there exists an equivariant map $\t K\to S^{d-1}$.

(b) The algorithmic problem of recognition almost embeddability of finite $k$-dimensional complexes in $\R^d$ is NP hard.}

The proof is based on the technique from the Matou\v{s}ek-Tancer-Wagner paper (proving an analogous result for
embeddings), and on singular versions of the higher-dimensional {B}orromean rings lemma. The new part of our argument is a stronger `almost embeddings' version of the generalized van Kampen--Flores theorem.
\end{abstract}

\tableofcontents


\section{Introduction}

In this paper we study almost embeddings and equivariant maps of configuration spaces.
They appear in studies of embeddings \cite{FKT, Sk08} as well as in topological
combinatorics (for Tverberg-type problems see \cite{BZ, BBZ, Sk16}).
Almost embeddings also turned out to be a useful tool for studying Helly-type
results on convex sets, implicitly in~\cite{Ma97} and explicitly
in~\cite{GoPaPaTaWa15}. See definitions and more motivations below.

Throughout this paper, let $K$ be a finite simplicial complex.

A map  $f\colon |K|\to \R^d$ is an {\bf almost embedding} if $f(\sigma)\cap f(\tau)=\emptyset$
whenever $\sigma,\tau$ are disjoint simplices of $K$.
(Existence of an almost embedding is obviously a necessary condition for existence of an embedding.)

The {\bf (simplicial) deleted product} of $K$ is
\[
\t K:= \cup \{ \sigma\times\tau\ : \sigma,\tau \textrm{ are  simplices of }K,\
\sigma\cap\tau=\emptyset\};
\]
i.e., $\t K$ is the union of products $\sigma\times\tau$ formed by disjoint simplices of $K$.

Suppose that $f:|K|\to\R^d$ is an almost embedding.
Then the map $\t f:\t K\to S^{d-1}$ is well-defined by the Gauss formula
$$\t f(x,y)=\frac{f(x)-f(y)}{|f(x)-f(y)|}.$$
We have $\t f(y,x)=-\t f(x,y)$;
i.e., this map is {\bf equivariant} with respect to the `exchanging factors' involution
$(x,y)\to(y,x)$ on $\t K$ and the antipodal involution on $S^{d-1}$.
Thus the existence of an equivariant map $\t K\to S^{d-1}$ is a
necessary condition for almost embeddability of $|K|$ in $\R^d$.

\begin{theorem}\label{t:nphard}
Fix integers $d,k\ge2$ such that $d=\frac{3k}2+1$.

\begin{itemize}
  \item[$(a)$] Assume that $P\ne NP$.
Then there exists a finite $k$-dimensional complex $K$ that does not admit an almost embedding in $\R^d$ but for which there exists an equivariant map $\t K\to S^{d-1}$.

\item[$(b)$] The algorithmic problem of recognition almost embeddability of finite $k$-dimensional complexes in
$\R^d$ is NP hard.
\end{itemize}
\end{theorem}

The reader need not to know what NP hardness is: the  essence of part (b) is explained by Theorem \ref{t:kphi} below.

For $k=2$ part (a) is true even without $P\ne NP$ assumption, by \cite[Theorem 1.5 and Proposition 1.7]{AMSW}.
Part (a) follows by part (b) and the existence of a polynomial algorithm for checking the existence of equivariant maps \cite{CKV}.
Indeed, for fixed $d,k$
it is polynomial time decidable whether there exists an equivariant map $\t K\to S^{d-1}$~\cite{CKV}.
Given that almost embeddabilty implies the existence of an equivariant map, part (b) implies part (a).

We discuss the possibility of removing the assumption $P\ne NP$
at the end of the introduction.

\begin{remark*}
The conclusions of Theorem~\ref{t:nphard} are in fact valid for each fixed integers $k,d$ such that
$2\le k\le d\le \frac{3k}2+1$ and $d\equiv 1 \pmod 3$.
That is, we reflect only the interesting extremal cases in the statement of Theorem~\ref{t:nphard}.

Indeed, for such integers $k,d$ let $k' := \frac{2(d-1)}3$.
For a proof of  part (a) with $(d, k)$ we take the complex $K'$ from part (a) with parameters $(d,k')$, and define
the complex $K$ to be the disjoint union of $K'$ and a $k$-simplex.
Similarly, for a proof of  part (b) with $(d, k)$ we add an isolated $k$-simplex to every $k'$-complex.
In both cases it is easy to check that the conclusion of Theorem~\ref{t:nphard} remains valid as for $d\ge k$ adding an isolated $k$-simplex does not affect neither almost embeddability to $\R^d$ nor the existence
of an equivariant map to $S^{d-1}$.
\end{remark*}

\paragraph{Motivation and background.}
A classical question in topology is to determine whether a simplicial complex $K$ embeds (topologically/piecewise linearly/linearly) in $\R^d$.
It is easy to deduce that every $k$-dimensional simplicial complex embeds (even
  linearly) into $\R^{2k+1}$.  Pioneering result in this area, known as the van
  Kampen--Flores theorem \cite{VK, Fl32, Sk14}, states the existence of
  $k$-dimensional complexes that to do not embed into $\R^{2k}$ (for every
  integer $k$; even topologically).

In general, it is often very hard to determine whether a given complex $K$ embeds into $\R^d$.
More precisely, this question subtly depends on the
comparison of $k:=\dim K$ and $d$.
For example, it is algorithmically undecidable to recognize whether a given $(d-1)$-complex embeds into $\R^d$~\cite{MaTaWa11}, provided that $d \geq 5$.
(This result follows from a celebrated theorem of Novikov on unrecognizability of the $d$-sphere~\cite{VoKuFo74}.)

Matou\v{s}ek, the second author and Wagner~\cite{MaTaWa11} proved that for each pair $(k,d)$ such that
$4\le d\le \frac{3n}2+1$ it is NP hard to decide whether a $k$-dimensional simplicial complex PL embeds in $\R^d$.
  Theorem~\ref{t:nphard}(b) is a version of this result for almost embeddability.

We describe the method from~\cite{MaTaWa11} in detail (in order to prove our main results), up to one step in proof that we take directly from~\cite{MaTaWa11}.
We explicitly state the initial step of the proof (Theorem~\ref{t:kphi} below).
Next, we slightly simplify the main construction (construction of $K(\Phi)$ in \S\ref{s:pr}).
We also present a simple proof of Lemma \ref{l:odd_obstruction} below generalizing the van Kampen-Flores theorem (which is also one of the key tools for the result).
For $\ell=k-1$ this lemma is proved in \cite{VK}, for $\ell<k-1$ a weaker version of this lemma
(when $f|_{S_1}$ is a PL embedding) is proved in \cite[Lemma~1.4]{SeSp92} using the Smith index.
Thus this paper can serve as an exposition of the proof of the above result of~\cite{MaTaWa11}.

Theorem~\ref{t:nphard}(b) is interesting on its own because almost embeddability is different from embeddability.
Consider the following three properties of a finite simplicial complex $K$.

(E) $K$ PL embeds into $\R^d$.

(AE) $K$ PL almost embeds in $\R^d$.

(EM) There exists an equivariant map $\t K\to S^{d-1}$.

The conditions (AE) and (EM) appeared as `combinatorial' or `algebraic' counterparts of (E), useful to study `geometric' condition (E).
Theorem \ref{t:nphard} indicates that the condition (AE) is closer to (E) than to (EM), from algorithmic point of view.
More precisely, we have
$$\xymatrix{ E \ar@{=>}[r]_{} & AE \ar@{=>}[r]_{}  & EM \ar@{=>}@(dl,dr)[ll]^{2d\ge3k+3\text{ or }d\le2}}$$
Here the straight arrows are clear and explained above, and the curved arrow is
a theorem of Weber~\cite{We67}; see also \cite[\S5]{Sk08}.
For every pair $(k,d)$ such that `$2d\ge3k+3$ or $d\le2$' does not hold, i.e. such that  $3\le d\le \frac{3k}2+1$, we have $(EM)\not\Rightarrow(E)$ \cite{SeSp92}, \cite{FKT}, \cite{SSS}, \cite{GS06}.
Moreover, for every pair $(k,d)$ such that $4\le d\le \frac{3k}2+1$
we have $(AE)\not\Rightarrow(E)$ \cite{SeSp92}, \cite{SSS}.
See \cite[\S5, \S7]{Sk08} for a survey.
By \cite[Theorem 1.5 and Proposition 1.7]{AMSW} $(EM)\not\Rightarrow(AE)$ for $d=2k=4$.
Theorem \ref{t:nphard}(a) shows (modulo $P\ne NP$) that for every pair $(k,d)$ such that $4\le d=\frac{3k}2+1$
$(EM)\not\Rightarrow(AE)$.

It might be also interesting to compare algorithmic complexity (of embeddability and almost embeddability) with the `geometric' \emph{refinement complexity} introduced in~\cite{FK}.
Although it is not directly related to the results in this paper, let us
also consider embeddability and almost embeddability of $k$-complexes to $\R^{2k}$ when $k \geq 3$. 
Then there is a quite noticeable gap between the two complexities for embeddability, which is polynomial time solvable whereas the refinement complexity grows exponentially~\cite{FK}. 
V. Krushkal kindly informed us that this gap is also present for almost
embeddability by generalizations of Proposition~4.1 and `Proof of the bound
(4.1) on refinement complexity' from~\cite{FK}.\footnote{However, it is not an
aim of this paper to provide the details.}

\paragraph{Proof technique.}
We prove Theorem~\ref{t:nphard} by applying the technique of \cite{MaTaWa11}
(which builds on a construction\footnote{Let us emphasize that~\cite{SeSp92},
\cite{FKT}, \cite{SSS}, and~\cite{MaTaWa11} slightly differ in technical
details. In particular, the examples in~\cite{SeSp92} and~\cite{SSS} were built with the aim to
be almost-embeddable, thus we could not use them immediately without a
modification.} in~\cite{SeSp92}, \cite{FKT}, \cite{SSS}; see \cite[\S5, \S7]{Sk08} for a survey) and
the Singular Borromean Rings Lemma \ref{l:bor} \cite{AMSW}.
The new part of our argument is the `almost embeddings' version (Lemma \ref{l:odd_obstruction}) of the generalized van Kampen--Flores theorem.
That is, of~\cite[Lemma 1.4]{SeSp92}, \cite[Lemma 6]{FKT},
\cite[Lemma 1.1]{SSS},~\cite[Lemma 7.2]{Sk08}, and~\cite[Lemmas~4.1 and ~5.1(i)]{MaTaWa11}.
Our version is stronger because we do not assume that $f|_{S_j}$ is an embedding.
In spite of this, our proof (presented in Section~\ref{s:elem}) is simpler, cf.  \cite[Remark at the end of 5.1]{MaTaWa11}.

Let us emphasize that this passage from embeddability to almost embeddability, being not hard, is not entirely trivial.
Although new proofs of `almost embeddings' analogues of main lemmas are simpler, they require certain change of the viewpoint.
The Borromean Rings Lemma \ref{l:bor} for almost embeddings is only proved for $k=2l$ \cite{AMSW},
not for $k\ge 2l$ as for embeddings.
Also recall that almost embeddability does not in general imply embeddability \cite{SeSp92, SSS}, cf. `Motivation and background'. 

\medskip
Formally, Theorem~\ref{t:nphard} follows by NP-hardness of recognition of 3-SAT problem and the following result.

A {\bf $3$-CNF formula} in variables $x_1, \dots, x_n$ is
$$\Phi(x_1,\ldots,x_n)=\bigwedge\limits_{s=1}^t
(x_{n_{s1}}^{\alpha_{s1}}\vee x_{n_{s2}}^{\alpha_{s2}}\vee x_{n_{s3}}^{\alpha_{s3}}).$$
Here $n_{si}\in\{1,\ldots,n\}$, $\alpha_{si}\in\{0,1\}$ and $x^0=\neg x$, $x^1=x$.

\begin{theorem}\label{t:kphi}
Let $d,k\ge2$ be fixed integers such that $d=\frac{3k}2+1$.
Then to each $3$-CNF formula $\Phi$ there corresponds, by a polynomial algorithm (in the size of the formula $\Phi$), a finite $k$-dimensional complex $K(\Phi)$ such that $K(\Phi)$ is almost embeddable in $\R^d$ if and only if $\Phi$ is satisfiable (i.e., if the Boolean function $\Z_2^n\to\Z_2$ corresponding to $\Phi$ is not identically zero).
\end{theorem}

The analogue of Theorem \ref{t:kphi} for embeddability is proved in~\cite[\S4, \S5]{MaTaWa11}.
Our complex $K(\Phi)$ slightly differs from the complex constructed in~\cite[\S4.2, \S5.2]{MaTaWa11},
which we call $K'(\Phi)$.%
\footnote{We make this minor change to simplify the construction of $K(\Phi)$ and proof of the `only if' part,
which, however, would work for $K'(\Phi)$ as well.}
For the `if' part of Theorem~\ref{t:kphi} we need the following lemma.

\begin{lemma}[proof is sketched in  Section \ref{s:pr}]\label{l:comp}
The complex $K(\Phi)$ constructed in Section \ref{s:pr} is obtained from the complex $K'(\Phi)$ constructed
in~\cite[\S4.2, \S5.2]{MaTaWa11} by several contractions of edges and several compressions of $S^{\ell-1}\times I$.
(Compression of a subcomplex identified with $S^{\ell-1}\times I$ is contracting to a point each segment $x\times I$, $x\in S^{\ell-1}$.)
\end{lemma}

\begin{proof}[Proof of Theorem~\ref{t:kphi}: the `if' part]
Since the formula $\Phi$ is satisfiable, the complex $K'(\Phi)$ PL embeds into $\R^d$ \cite[Section 5]{MaTaWa11}.%
\footnote{The converse is also true but is not used here.}
Recall that the quotient of a PL manifold by a map with collapsible point-inverses is PL homeomorphic to the same manifold \cite{Co67}.
Hence contracting an edge and
compressing $S^{\ell-1}\times I$ keep PL embeddability.%
\footnote{The analogous claim is not true for `decontractions'.}
Thus by Lemma \ref{l:comp} $K(\Phi)$ also PL embeds into $\R^d$.
\end{proof}

Therefore it suffices to prove the converse: if $K(\Phi)$ almost-embeds in $\R^d$, then $\Phi$ is satisfiable.
This is a strengthening of the analogous fact from~\cite{MaTaWa11}.
We use the same idea as~\cite{MaTaWa11} but need to replace two key lemmas~\cite[Lemmas~5.1(i) and~5.3(i)]{MaTaWa11}
by suitable analogues for almost embeddings.
These analogues are Lemma~\ref{l:5.1} and the Singular Borromean Rings Lemma \ref{l:bor} \cite[Lemma 1.9]{AMSW}
below, respectively.

Let  $T:=S^\ell\times S^\ell$ be the $2\ell$-dimensional torus with meridian $a:=S^\ell\times \cdot$ and parallel $b:=\cdot\times S^\ell$.
See well-known definition of `linked modulo 2', e.g.,
in~\cite[\S77]{SeTh80} or in~\cite[\S2.2 `Linking modulo 2']{Sk}.

\begin{lemma}[Singular Borromean Rings]\label{l:bor}
  For each $k=2\ell$ let $S^k_a$ and $S^k_b$ be copies of $S^k$.
  Then there is no PL map $f\colon T\sqcup S^k_a\sqcup S^k_b\to \R^{k+\ell+1}$
  such that

  (a) the $f$-images of the components are pairwise disjoint;

  (b) $f(S^k_a)$ is linked modulo 2 with $f(a)$ and is not linked modulo 2 with $f(b)$;

  (c) $f(S^k_b)$ is linked modulo 2 with $f(b)$ and is not linked modulo 2 with $f(a)$.
\end{lemma}

\paragraph{On the assumption $P \neq NP$.}
A reader could expect that analyzing the algorithm in~\cite{CKV} and the proof of Theorem~\ref{t:nphard}(b) would yield a direct construction of an example of Theorem~\ref{t:nphard}(a), without the assumption $P\neq NP$.
{Here we discuss the difficulties that appear in this analysis.

Let us consider a 3-CNF formula $\Phi$ and let $d$ and $k$ be fixed, $d = \frac{3k}2 + 1$.
If $\Phi$ is satisfiable, then $(AE)$ holds for $K(\Phi)$ by Theorem~\ref{t:kphi} and therefore $(EM)$ holds for
$K(\Phi)$ as well.

If $\Phi$ is not satisfiable, then $(AE)$ does not hold for $K(\Phi)$ by
Theorem~\ref{t:kphi} but we do not know whether $(EM)$ holds for $K(\Phi)$.
However, if $(EM)$ did not hold for every $\Phi$ which is not satisfiable, then
we would deduce that it is NP-hard to recognize whether a given simplicial
complex $K$ satisfies $(EM)$. On the other hand, this is a polynomial time
solvable problem due to~\cite{CKV}. This would be only possible if $P = NP$.

Thus, we have good reasons to expect that for any choice of $k$ and $d$ with $d
= \frac{3k}2 + 1$, there is a non-satisfiable formula $\Phi$ such that
$(EM)$ holds for $K(\Phi)$. Actually, we conjecture that $(EM)$ holds for $K(\Phi)$ for every $\Phi$.

For a proof of Theorem~\ref{t:nphard}(a), without the assumption $P \neq NP$,
it would be fully sufficient to exhibit a single non-satisfiable $3$-CNF
formula $\Phi$ such that $(EM)$ holds for $K(\Phi)$. In fact, the construction in
Theorem~\ref{t:kphi} makes also sense for the simplest non-satisfiable $1$-CNF
formula $\Phin = x_1 \wedge \neg x_1$ (in the definition of the clause gadget
$G$ in Section~\ref{s:pr} only a single simplex is removed instead of three
simplices). Let $\Kn(k,d) := K(\Phin)$ for given
parameters $d$ and $k$ (here we want to emphasize the dependence on $k$ and
$d$). The complex $\Kn(2,4)$ is essentially the complex constructed by Freedman,
Krushkal and Teichner~\cite{FKT} (up to a minor modification), and we know that
$(EM)$ holds in this case (as we discussed below the statement of
Theorem~\ref{t:nphard}).

For few other values of $k$ and $d$ we could, in principle, run the algorithm
of~\cite{CKV} on $\Kn(k,d)$ (unfortunately, it is not implemented\footnote{That
is, the code is not written.}).
However, we do not know how to verify $(EM)$ for infinitely many values of $k$ and $d$ we
are interested in; the dependence of the algorithm from~\cite{CKV} on $k$ and
$d$ is somewhat complicated. The algorithm in~\cite{CKV} is based on
the obstruction theory. As far as we know there are no other tools developed,
besides the obstruction theory, that would allow us to verify $(EM)$ for our
examples `by hand'.


\section{Proof of Theorem~\ref{t:kphi}}\label{s:pr}

First we define building blocks for the complex $K(\Phi)$ (most importantly, clause gadgets) and prove their properties.

Take any integers $0\le \ell<k$.
We suppress dependence on $k, \ell$.
For an integer $n$ denote
$$[n]:=\{1, \ldots, n\}.$$

\paragraph{Definition of an auxiliary complex $F$.}
Complex $F$ has the vertex set
\linebreak
$[k +\ell + 3]\cup\{p\}$.
The simplices are
\begin{itemize}
\item complete $k$-skeleton on $[k +\ell + 3]$, and
\item all the simplices of dimension at most $\ell + 1$ that contain $p$.
\end{itemize}
In other words,
$$F := \left(\ [k +\ell + 3]\cup \{p\}\ ,
\ {[k +\ell+3]\choose\le k+1}\cup\left\{\{p\}\cup\sigma\ :\ \sigma\in{[k +\ell+3]\choose\le\ell+1}\right\}\ \right).$$
Here ${n\choose \le m}$ is the set of all subsets of $[n]$ having at most $m$ elements.

\paragraph{Definitions of $\sigma_j$, $S_j$ and  clause gadget $G$.}
In this definition $j$ is any element of $[3]$.

Set $\sigma_j$ to be the simplex with vertex set $\{p\}\cup[\ell+2]-\{j\}$.
Then $\sigma_1$, $\sigma_2$ and $\sigma_3$ are three $(\ell+1)$-simplices containing $p$.

Set $S_j$ to be the union of all $k$-simplices with vertices in $[k +\ell + 3]$ that do not intersect $\sigma_j$.
Clearly, this union is homeomorphic to the $k$-sphere.%
\footnote{This was called a {\it complementary} sphere in \cite{MaTaWa11}.}

Finally, we define the \emph{clause gadget} as $G:=F-\sigma_1-\sigma_2-\sigma_3$.


\begin{lemma}\label{l:5.1}
For any $\ell<k$ and general position PL almost-embedding $f:|G|\to \R^{k +\ell + 1}$ there is $i\in[3]$ such that $f(\partial\sigma_i)$ is linked modulo 2 with $f(S_i)$.
\end{lemma}

For $k=2\ell=2$ Lemma \ref{l:5.1} is proved in \cite[proof of Lemma 8]{FKT}.
Cf. \cite[Remark 2.3.b]{AMSW}.

\begin{proof}[Proof of Theorem~\ref{t:kphi}: construction of $K(\Phi)$]
Recall the notation for $3$-CNF formula $\Phi$ given before Theorem~\ref{t:kphi}.
The  `multiple' $x_{n_{s1}}^{\alpha_{s1}}\vee x_{n_{s2}}^{\alpha_{s2}}\vee x_{n_{s3}}^{\alpha_{s3}}$ is the $s$-th \emph{clause} of $\Phi$.
The   `summand' $x_{n_{si}}^{\alpha_{si}}$  is the $i$-th \emph{literal} of the $s$-th clause.
Without loss of generality, we may assume that no clause (`multiple') contains both $x_m$ and $\neg x_m$ for some $m$
(otherwise such a clause would be redundant).
Denote
$$P:=\{(q,r)\in([t]\times[3])^2\ :\ n_q=n_r,\ \alpha_q=0\text{ and }\alpha_r=1\}.$$
This is the set of all pairs $(q,r)=((q_1,q_2),(r_1,r_2))$ such that for some $m$

$\bullet$ the $q_2$-th literal (`summand') of the $q_1$-th clause (`multiple') is $x_m$,  and

$\bullet$ the $r_2$-th literal (`summand') of the $r_1$-th clause (`multiple') is $\neg x_m$.

(In other words, this is the set of all pairs of (pairs of) indices of literals in conflict.)

Take copies $G_1,\ldots,G_t$ of (the clause gadget) $G$.
Denote by $\sigma_q=\sigma_{(q_1,q_2)}$ the simplex $\sigma_{q_2}$ in the copy $G_{q_1}$ .
Take a triangulation of $k$-torus $T$ extending triangulations of its meridian and parallel $a$ and $b$ as boundaries of $(\ell+1)$-simplices.
For each $(q,r)\in P$ take a copy $T_{qr}\supset a_{qr},b_{qr}$ of $T\supset a,b$.
Set
$$K(\Phi):=\cup_{s=1}^t G_s
\bigcup\limits_{\partial\sigma_q=a_{qr},\ \partial\sigma_r=b_{qr}, \ (q,r)\in P}
\cup_{(q,r)\in P} T_{qr}.$$
That is, this complex is obtained from the copies $G_s$ and $T_{qr}$, by identifying the $\ell$-spheres $\partial\sigma_q$ and   $a_{qr}$, and the $\ell$-spheres $\partial\sigma_r$ and $b_{qr}$, for each $(q,r)\in P$.

Recall that $k$ and $\ell$ are fixed.
Then each of the complexes $G$ and $T$ can be built by a constant-time algorithm.
Hence $K(\Phi)$ is obtained from $\Phi$ by a polynomial algorithm in $n$ and $t$ (i.e. in the size of the formula).
\end{proof}

\begin{proof}[Sketch of a proof of Lemma \ref{l:comp}]
The construction of $K'(\Phi)$ \cite[\S4.2, \S5.2]{MaTaWa11} is different from the above construction of $K(\Phi)$
by the following details:

$\bullet$ the torus $T$ is replaced by a polyhedron
$X$ containing `parallel' and `meridian' $a$ and $b$, and an edge whose contraction yields $T$;

$\bullet$ the simplices $\sigma_i$ in the definitions of $G$ and of $K(\Phi)$ are replaced by $k$-disks $\omega_i\subset\Int\sigma_i$.

Thus $K(\Phi)$ is obtained from $K'(\Phi)$ by contracting the edge in each copy of $T$, and by compressing $\sigma_j-\Int\omega_j\cong S^{\ell-1}$ in each copy of $G$ and for each $j\in[3]$.
\end{proof}

\begin{proof}[Proof of Theorem~\ref{t:kphi}: the `only if' part]
Assume that there is a PL almost embedding $f:|K(\Phi)|\to \R^{k +\ell + 1}$.
By Lemma~\ref{l:5.1} for every $s\in[t]$ there is $i(s)\in[3]$ such that
the $f$-images of the spheres  $S_{si(s)},\partial\sigma_{si(s)}\subset |G_s|$ are linked modulo 2.

Let us assume, for contradiction, that $\Phi$ is not satisfiable; that is,
  $\Phi \equiv 0$. (The reader not so familiar with literals, clauses and conflicts,
may wish to skip the next paragraph and check rather the footnote in the following one.)

  The function $i$ selects one literal (`summand') in each clause (`multiple').
  Then two selected literals (`summands') must be in conflict.

This means that there are $q_1,r_1\in[t]$ and $m\in[n]$ such that the $i(q_1)$-th literal
  (`summand') of the $q_1$-th clause (`multiple') is $x_m$ and the $i(r_1)$-th literal
  (`summand') of the $r_1$-th clause (`multiple') is $\neg x_m$.%
\footnote{Indeed, in the opposite
  case for each $m\in [n]$ there is $\alpha(m)\in\{0,1\}$
  such that $\alpha_{si(s)}=\alpha(n_{si(s)})$ for each $s\in[t]$.
  So we can take $x_{n_{si(s)}}:=\alpha(n_{si(s)})$  for each $s\in[t]$ and
  then extend this to a satisfying assignment.}
That is, $(q,r):=((q_1,i(q_1)),(r_1,i(r_1)))\in P$.

Then $\partial\sigma_q=a_{qr}$ and $\partial\sigma_r=b_{qr}$.
Since $f$ is an almost embedding, the $f$-images of  $S_q$, $S_r$ and $T_{qr}$  are pairwise disjoint.
The $\ell$-sphere $b_{qr}$ bounds the disk $v*b_{qr}$ outside $S_q$, where $v$ is any vertex of $G_{r_1}$ outside $\partial \sigma_{r_1i}$ for each $i \in [3]$.
Hence $f(S_q)$ is unlinked modulo 2 with $f(b_{qr})$.
Analogously $f(S_r)$ is unlinked modulo 2 with $f(a_{qr})$.
Since $k=2\ell$, all this contradicts the Singular Borromean Rings Lemma~\ref{l:bor}
applied to the restriction of $f$ to $S_q\sqcup S_r\sqcup T_{qr}$.
\end{proof}

\section{The van Kampen number: proof of Lemma~\ref{l:5.1}}\label{s:elem}

For a  general position PL map $f\colon |K|\to \R^d$  of a finite $k$-complex define {\it the van Kampen number}
$$v(f)\in\Z_2$$
to be the parity of the number of points $x\in \R^d$ such that $x\in f(\sigma) \cap f(\tau)$ for some disjoint simplices $\sigma, \tau \in  K$ with $\dim \sigma + \dim \tau = d$.
(For an exposition and another applications of the van Kampen number see
\cite{Sk17}, \cite[\S1]{Sk}.)

\begin{lemma}\label{l:pa}
Let $d$ be an integer and $K$ a finite complex such that for every pair $\sigma, \tau$ of disjoint $s$- and $t$-simplices in $K$ with $s+t=d-1$ the following two numbers have the same parity:

$\bullet$ the number of $(s+1)$-simplices $\nu$ containing $\sigma$ and disjoint with $\tau$;

$\bullet$ the number of $(t+1)$-simplices $\mu$ containing $\tau$ and disjoint with $\sigma$.

Then $v(f)$ is independent of a general position PL map $f\colon |K|\to \R^d$.
\end{lemma}

For $d=2$ and $K=K_5$ this corresponds to well-known proof of the non-planarity of $K_5$ \cite[Lemma 3.4]{Sk}, \cite[\S5]{BE82}.
For the general case the proof is analogous.

\begin{proof}[Proof of Lemma~\ref{l:pa}] Lemma~\ref{l:pa} follows analogously to \cite[Lemma 3.5]{Sh57} (by interpreting $v(f)$ as an obstruction to the existence of certain equivariant map).

A direct proof is as follows.
Take a general position PL homotopy $H:|K|\times I\to\R^d\times I$ between general position PL maps $H_0,H_1:|K|\to\R^d$.
Then
$$v(H):=\cup\{H(\sigma\times I)\cap H(\tau\times I)\ :\ \sigma,\tau\in K,\ \sigma\cap\tau=\emptyset,\ \dim\sigma+\dim\tau=d\}$$
is a graph.
For each $i=0,1$ the vertices of this graph in $\R^d\times i$ are exactly the points $x$ from the definition of $v(H_i)$.
So $v(H_i)$ equals to the number modulo 2 of vertices of this graph in $\R^d\times i$.
This graph also has vertices in $\R^d\times(0,1)$ corresponding to pairs $(\nu,\tau)$ and $(\sigma,\mu)$
from the bullet points of the lemma.
(There could be connected components of $v(H)$ containing no such vertices;
there could be some other vertices of even degree, e.g. vertices of degree 2 coming from double points of $H$ or
vertices of degree 4 coming from triple points of $H$.)

Analogously to \cite[Lemma 11.4]{Hu69}  any vertex of this graph

$\bullet$ contained in $\R^d\times i$ has odd degree.

$\bullet$ contained in $\R^d\times(0,1)$ has even degree (by the bullet points of the lemma).

Hence $v(H_0)=v(H_1)$.
\end{proof}

\begin{lemma}
\label{l:odd_obstruction}
For any $\ell<k$ and general position PL map $f\colon |F| \to \R^{k + \ell + 1}$ we have $v(f)=1\in\Z_2$.
\end{lemma}

\begin{proof}
For {\it some} $f$ Lemma \ref{l:odd_obstruction} was proved in \cite[Lemma~1.1]{SeSp92}.
Then Lemma \ref{l:odd_obstruction} follows for {\it any} $f$ by Lemma \ref{l:pa} after we verify the assumptions
of the lemma.

Take any pair $\sigma, \tau$ of disjoint $s$- and $t$-simplices in $F$ such that $s+t=k + \ell$.
Without loss of generality we assume that $s\le t$.
Since $t\le k$, we obtain $s\ge\ell$.
Among $(k+1)+(\ell+1)+2$ vertices of $F$ there are exactly two which are not contained in $\sigma\cup\tau$.
We distinguish two cases.

\begin{itemize}
\item {\it Case $s=\ell$.} In this case $t=k$.
Since $\dim F = k$, the simplex $\tau$ cannot be extended to a simplex of $F$ (disjoint with $\sigma$).
Since $F$ contains complete $(\ell + 1)$-skeleton, $\sigma$ can be extended (to $(s+1)$-simplex of $F$ disjoint with $\tau$) by both vertices of $F$ not contained in $\sigma\cup\tau$.
The numbers $0$ and $2$ have the same parity as required.

\item {\it Case $s>\ell$.} In this case $\ell <t< k$.


{\it Subcase when neither $\sigma$ nor $\tau$ contains $p$.} Since $p$ is contained only in simplices of dimension at most $\ell + 1$, neither $\sigma$ not $\tau$ can be extended by $p$ to a simplex of $F$.
On the other hand, since $s,t< k$, the remaining vertex of $F$ can serve for extension of both $\sigma$ or $\tau$.
The numbers $1$ and $1$ have the same parity as required.

{\it Subcase when $\sigma$ or $\tau$ contains $p$.} Since $t>\ell$ and $p$ can be only contained in a simplex of dimension at most $\ell + 1$, we can without loss of generality assume that $p\in\sigma$ and $s= \ell + 1$.
Since $p$ does not belong to any $(\ell +2)$-simplex, it follows that $\sigma$ cannot be extended.
On the other hand, $\tau$ can be extended in two ways to both vertices of $F$ not contained in $\sigma\cup\tau$.
The numbers $0$ and $2$ have the same parity as required.
\end{itemize}
\end{proof}

\begin{proof}[Proof of Lemma~\ref{l:5.1}]
Extend the map $f$ to a general position PL map $g\colon |F| \to \R^{k+\ell+1}$.
Since $\ell+2k<2(k+\ell+1)$, by general position to every point $x$ from the definition of $v(g)$ there corresponds a unique unordered pair of simplices of $F$, the sum of whose dimensions is $d$ and the intersection of whose $g$-images contains $x$.
Since $f$ is an almost-embedding, for every such point $x$ there is a unique $i\in[3]$ such that
$x\in g(\sigma_i)\cap g(S_i)$.
This and Lemma~\ref{l:odd_obstruction} imply that
$\sum_{i=1}^3|g(\sigma_i)\cap g(S_i)|\underset2\equiv v(g)=1\in\Z_2$.
Hence one of the three summands is odd as required.
\end{proof}


\bibliographystyle{alpha}
\bibliography{ae}

\end{document}